\documentclass[twoside,11pt]{amsart}

\usepackage{a4wide}
\usepackage{latexsym}
\usepackage{mathrsfs}
\usepackage[T1]{fontenc}
\usepackage{enumitem}
\usepackage{mathrsfs}

\usepackage{amssymb}
\usepackage{latexsym}
\usepackage{amsmath}
\usepackage{fancyhdr}
\usepackage{bbm}
\usepackage{setspace}
\usepackage{mathabx}

\numberwithin{equation}{section}

\newtheorem{theorem}{Theorem}[section] 
\newtheorem{lemma}[theorem]{Lemma}     
\newtheorem{corollary}[theorem]{Corollary}
\newtheorem{proposition}[theorem]{Proposition}
\theoremstyle{definition}
    
\newtheorem{conjecture}[theorem]{Conjecture}  
\newtheorem{definition}[theorem]{Definition}
\newtheorem{example}[theorem]{Example}
\newtheorem{question}[theorem]{Question}
 
\newcommand {\C}{\mathbb C}
\newcommand {\T}{\mathbb T}
\newcommand {\Q}{\mathbb Q}
\newcommand {\R}{\mathbb R}
\newcommand {\N}{\mathbb N}
\newcommand {\Z}{\mathbb Z}

\newcommand {\B}{\mathcal B}

\newcommand{\supp}{{\rm supp\,}}

\newcommand{\1}{\mathbbm{1}}
\newcommand{\im}{\operatorname{im}}
\newcommand{\id}{\operatorname{id}}
\newcommand{\dd}{{\rm \, d}}

\newcommand{\ind}{\operatorname{ind}}
\newcommand{\pkr}{\operatorname{pker}}
\newcommand{\e}{{\rm e}}

\allowdisplaybreaks

\title[The Ideal Structure of Measure Algebras]{The ideal structure of measure algebras and asymptotic properties of group representations}


\date{2021}


\begin{document}

\author[J.\ T.\  White]{Jared T.\ White}
\address{
	Jared T. White, 31 Warham Road, London N4 1AR, United Kingdom}
\email{jared.white@open.ac.uk}
	
	\date{2021}

\maketitle	
	
	\begin{center}
	\textit{The Open University}
\end{center}

\begin{abstract}
	We classify the weak*-closed maximal left ideals of the measure algebra $M(G)$ for certain Hermitian locally compact groups $G$ in terms of the irreducible representations of $G$ and their asymptotic properties. In particular, we obtain a classification for connected nilpotent Lie groups, and the Euclidean rigid motion groups. We also prove a version of this result for certain weighted measure algebras. We apply our classification to obtain an analogue of Barnes' Theorem on integrable representations for representations vanishing at infinity. We next study the relationship between weak*-closedness and finite generation, proving that in many cases $M(G)$ has no finitely-generated maximal left ideals. We also show that $M(\R^2 \rtimes SO(2))$ has a weak*-closed maximal left ideal that is not generated by a projection, and investigate whether or not it has any weak*-closed left ideals which are not finitely-generated.
\end{abstract}

\section{Introduction}
\noindent
In this article we shall study the weak*-closed ideal structure of the measure algebra of a locally compact group, and use the understanding gained to prove new results about representations vanishing at infinity and about finitely-generated left ideals.

In \cite{W3} the author proved that for a compact group $G$ there is a bijective correspondence between the closed left ideals of $L^1(G)$ and the weak*-closed left ideals of $M(G)$, and used this to give a classification of the weak*-closed left ideals of $M(G)$ in terms of the representation theory of $G$. In the present article we seek similar results that go beyond the compact setting. For a non-compact locally compact group $G$ there is no description available of the closed left ideals of $L^1(G)$ (and finding any description seems to be an intractable problem), and as such we shall focus on maximal left ideals.
For a Hermitian locally compact group $G$ it is known that the maximal modular left ideals of $L^1(G)$ all have the form $\mathcal{I}_{\pi, \xi} := \{ f \in L^1(G) : \pi(f)\xi = 0 \}$, for some irreducible representation $\pi$ and some unit vector $\xi \in \mathcal{H}_\pi$ (although it is unclear exactly which left ideals of this form are maximal modular).  We shall study the following conjecture. We define $\mathcal{J}_{\pi, \xi}$ to be the closed left ideal of the measure algebra given by 	$\{ \mu \in M(G) : \pi(\mu) \xi = 0 \}$.

\begin{conjecture}		\label{0.1}
	Let $G$ be a Hermitian locally compact group. The weak*-closed maximal left ideals of $M(G)$ are given exactly by $\mathcal{J}_{\pi, \xi}$,
	where $\pi$ is an irreducible representation vanishing at infinity, and $\xi$ is a unit vector in $\mathcal{H}_\pi$ with the property that $\mathcal{I}_{\pi,\xi}$ is a maximal modular left ideal of $L^1(G)$.
\end{conjecture}

Our main result is Theorem \ref{4.4a}, which states that our conjecture is true for a large class of Hermitian locally compact groups, which includes connected nilpotent Lie groups, the Euclidean rigid motion groups $\R^n \rtimes SO(n)$, and the Fell groups $\Q_p \rtimes \mathcal{O}_p$. We shall also prove in Theorem \ref{4.3f} that certain weighted measure algebras enjoy a bijective correspondence between their weak*-closed maximal left ideals, and the maximal modular left ideals of the weighted $L^1$-algebra (without any constraint on the representation), similar to the compact case.

We shall now explain some of the applications of our classification result. In his seminal work of 1980 \cite{B80}, Barnes established a correspondence between integrable irreducible representations of a unimodular locally compact group, and minimal projections in $L^1(G)$. In this context ``minimal'' means that the left ideal generated by the projection is a minimal left ideal, which is a stronger property than being minimal with respect to the partial order on projections. Another important asymptotic property of a representation of a locally compact group is that of vanishing at infinity. In this article we make a conjecture relating weak*-closed maximal left ideals of the measure algebra $M(G)$ of a Hermitian locally compact group $G$ and irreducible representations vanishing at infinity that we believe should be thought of as the correct analogue of Barnes' Theorem in this setting.

\begin{conjecture}		\label{0.2}
	Let $G$ be a Hermitian locally compact group, and let $\pi \in \widehat{G}$. 	
	Then there exists a unit vector $\xi \in \mathcal{H}_\pi$ for which $\mathcal{J}_{\pi, \xi}$ is a weak*-closed maximal left ideal of $M(G)$ if and only if $\pi$ vanishes at infinity.	
\end{conjecture}

 In Proposition \ref{4.11} we shall rephrase Barnes' Theorem in terms of maximal left ideals of the measure algebra, so that the analogy with our conjecture is more transparent. In Theorem \ref{4.15} we shall prove the conjecture for those groups satisfying the hypothesis of Theorem \ref{4.4a} that are additionally CCR (which still includes connected nilpotent Lie groups and the Euclidean rigid motion groups). 
 
 Another application is to the theory of finitely-generated maximal left ideals of measure algebras and group algebras, a topic previously studied by the author in \cite{W1}. It was proved there that, when $G$ is not discrete, $L^1(G)$ has no finitely-generated maximal modular left ideals, but the case of discrete $G$ remains open. The connection with our classification result is that
 finitely-generated norm-closed left ideals of a measure algebra are automatically weak*-closed.
 A consequence of our investigations shall be that, for some locally compact groups, the measure algebra has no finitely-generated maximal left ideals at all. In particular, when $G$ is finitely-generated and virtually nilpotent $L^1(G) = M(G)$ has no finitely-generated maximal (modular) left ideals, just like for non-discrete groups.

\begin{theorem}		\label{0.3}	
		Let $G$ be a locally compact group which is any of the following:
		\begin{enumerate}
			\item[\rm (i)] a non-compact Moore group;
			\item[\rm (ii)] a group with non-compact centre;
			\item[\rm (iii)] an infinite finitely-generated virtually nilpotent group;
			\item[\rm (iv)] a group of the form $\R^2 \rtimes_A \R$, for some $A \in M_2(\R)$;
		\end{enumerate}
		Then $M(G)$ has no weak*-closed, and hence no finitely-generated, maximal left ideals.
\end{theorem}

In case (iv), $\R$ acts on $\R^2$ via $t \colon x \rightarrow {\rm e}^{tA}x \ (t \in \R, \ x \in \R^2).$

We now discuss our final topic. We wish to better understand the relationship between the following properties of a closed left ideal of a measure algebra: (1) being generated by a projection, (2) being finitely-generated, and (3) being weak*-closed. Each of these notions is implied by the previous one, and it would be interesting to know whether or not they are all distinct. We are able to partially resolve this question by looking at the two-dimensional Euclidean rigid motion group in detail. 

\begin{theorem}		\label{0.4}
	Let $G = \R^2 \rtimes SO(2)$, and let $U = \ind_{\R^2}^G \chi$, where $\chi$ is the character on $\R^2$ given by $\chi(x) = \e^{ix_1} \ (x \in \R^2)$. Then
	\begin{enumerate}
		\item[\rm (i)] There exists a unit vector $\xi \in L^2(SO(2))$ such that $\mathcal{J}_{U, \xi}$ is a weak*-closed maximal left ideal that is not generated by a projection.
		\item[\rm (ii)] The left ideal $\mathcal{J}_{U,1}$ is weak*-closed, and is not generated by finitely many compactly supported measures.
	\end{enumerate}
\end{theorem}

We have been unable to resolve whether $\mathcal{J}_{U,1}$ is maximal, and whether it is finitely-generated.

Let us compare the previous result with the situation when $G$ is compact. In that case it follows from the classification theorem \cite[Theorem 6.2]{W3} that the weak*-closed maximal left ideals of $M(G)$ are given exactly by $\mathcal{J}_{\pi, \xi}$, for $\pi \in \widehat{G}$ and $\xi \in \mathcal{H}_\pi$ of norm $1$. Moreover, it is routinely verified that $\mathcal{J}_{\pi, \xi}$ is equal to $M(G)*(\delta_e - p)$, where $p$ is the projection in $L^1(G)$ defined by $p(t) = \frac{1}{\dim \mathcal{H}_\pi} \langle \pi(t^{-1}) \xi, \xi \rangle \ (t \in G)$.
In particular, in this setting, a maximal left ideal of $M(G)$ is weak*-closed if and only if it is generated by a projection.

The paper is organised as follows. In Section 2 we recall some background and fix our notation. In Section 3 we prove that $M(G)$ cannot have a proper weak*-closed left ideal of finite codimension unless $G$ is compact, and derive some consequences of this. We begin Section 4 by proving some general results that relate maximal modular left ideals of a Banach algebra $A$ to the weak*-closed left ideals of its multiplier algebra $M(A)$ in the special case that $M(A)$ is a dual Banach algebra. We then use these results to prove our main classification result, Theorem \ref{4.4a}, which establishes Conjecture \ref{0.1} in many cases, as well as Theorem \ref{4.3f}, which addresses the weighted setting. In Section 5 we discuss Barnes' Theorem and prove Conjecture \ref{0.2} for many Hermitian locally compact groups. Next, in Section 6, we give a detailed description of the weak*-closed maximal left ideals for some chosen examples. Section 7 is concerned with proving Theorem \ref{0.3} and Theorem \ref{0.4}. Finally, in Section 8 we list some open problems.

\section{Background and Notation}
\noindent
For us $\N = \{ 1,2,\ldots \}$. Let $X$ be a set, and let $Y \subset X$ be any subset. We denote by $\1_Y \colon X \rightarrow \{ 0, 1 \}$ the characteristic function of $Y$.

Let $E$ be a Banach space. We define $B_E$ to be the unit ball of $E$ and $S_E$ to be the unit sphere. We denote the dual space of $E$ by $E'$, and given a bounded linear map $\varphi \colon E \rightarrow F$ between Banach spaces, we write $\varphi' \colon F' \rightarrow E'$ for the dual map.
Given subsets $X \subset E$ and $Y \subset E'$ we shall write
$$X^\perp = \{ \lambda \in E' : \langle x, \lambda \rangle = 0 \ (x \in X) \} \qquad \text{and} \qquad Y_\perp = \{ x \in E : \langle x , \lambda \rangle = 0 \ (\lambda \in Y) \}.$$

Let $A$ be an algebra, and let $I$ be a left ideal of $A$. We say that $u \in A$ is a \textit{modular right identity} for $I$ if $a - au \in I$ for all $a \in A$. In the case that $I$ has a modular right identity we say that $I$ is a \textit{modular left ideal}. By a maximal modular left ideal we mean a maximal left ideal that is modular, or equivalently a left ideal that is maximal subject to being modular.  We say that $A$ is \textit{semisimple} if the intersection of the maximal modular left ideals of $A$ is $\{ 0 \}$.

Let $G$ be a locally compact group. We write $\int \cdots \dd t$ for integrals with respect to the left Haar measure on $G$, which we always take to be normalised when $G$ is compact.

By a \textit{representation of $G$} we shall always mean a strongly continuous unitary representation of $G$ on a Hilbert space. Typically a representation will be denoted by $(\pi, \mathcal{H}_\pi)$, where $\pi$ is a homomorphism from $G$ to the group of unitaries on the Hilbert space $\mathcal{H}_\pi$.  We shall write $\widehat{G}$ for the unitary dual of $G$, that is the collection of irreducible representations of $G$ up to unitary equivalence. Given a representation $(\pi, \mathcal{H}_\pi)$, and $\xi, \eta \in \mathcal{H}_\pi$ we shall write $\xi*_\pi \eta$ for the function $G \rightarrow \C$ given by $(\xi*_\pi \eta) (t) = \langle \pi(t) \xi, \eta \rangle \ (t \in G)$. We shall also write $\pkr \pi = \{ t \in G: \pi(t) \in \C \id_{\mathcal{H}_\pi} \}$, the \textit{projective kernel of $\pi$}, which is a closed normal subgroup of $G$. Note that, given vectors $\xi, \eta$, the function $G/\pkr \pi \rightarrow [0, \infty)$ given by $t \pkr \pi \mapsto |\langle \pi(t)\xi, \eta \rangle| \ (t \in G)$ is well-defined and continuous, and we say that $\pi$\textit{ vanishes at infinity modulo its projective kernel} if this function belongs to $C_0(G/\pkr \pi)$ for each $\xi, \eta \in \mathcal{H}_\pi$.

We say that a group $G$ is \textit{CCR} if, for every $\pi \in \widehat{G}$ and every $f \in L^1(G)$, the operator $\pi(f)$ is compact; this is equivalent to the Fell topology on $\widehat{G}$ being $T_1$. See \cite[Chapter 7]{F} for more details.

Let $(\pi, \mathcal{H}_\pi)$ be an irreducible representation of the locally compact group $G$. We recall the definitions of the following asymptotic properties that $\pi$ could have.

\begin{definition} 	\label{0.5}
	We say that
	\begin{enumerate}
		\item[\rm (i)] $\pi$ \textit{vanishes at infinity} if, for every $\xi, \eta \in \mathcal{H}_\pi$ we have $\xi *_\pi \eta \in C_0(G)$;
		
		\item[\rm (ii)] $\pi$ is \textit{integrable} if $\xi*_\pi \eta \in L^1(G)$ for some $\xi, \eta \in \mathcal{H}_\pi \setminus \{ 0 \}$;
		
		%\item[\rm (iii)] $\pi$ is \textit{square-integrable} if $\xi*_\pi \eta \in L^2(G)$ for some $\xi, \eta \in \mathcal{H}_\pi \setminus \{ 0 \}$. 
	\end{enumerate}
\end{definition}

It is known that an irreducible representation $\pi$ vanishes at infinity if and only if there exist some $\xi, \eta \in \mathcal{H}_\pi \setminus \{ 0 \}$ such that $\xi*_\pi \eta \in C_0(G)$. For background on asymptotic properties of representations see \cite{D77}.

The main focus of this article will be the measure algebra of a locally compact group, but we shall also prove a couple of results in the weighted setting, which we recall now. By a \textit{weight} on $G$ we mean a continuous function $\omega \colon G \rightarrow [1, \infty)$ such that $\omega(e) = 1$, and $\omega(st) \leq \omega(s) \omega(t) \ (s,t \in G)$. Many authors do not insist that weights need to be continuous, but we find this assumption convenient for the present article. We then define $$L^1(G,\omega) = \left\{ f \in L^1(G) : \int_G \vert f(t) \vert \omega(t) \dd t < \infty \right\},$$ 
and 
$$M(G, \omega) = \left\{ \mu \in M(G) :\int_G \omega(t) \dd \vert \mu \vert (t) < \infty \right\}.$$
We define a norm on $M(G, \omega)$ by $\Vert \mu \Vert_\omega = \int_G \omega(t) \dd \vert \mu \vert (t)$, and this makes it into a Banach algebra, with multiplication given by convolution. It has a  predual given by 
$$C_0(G, 1/\omega) = \{ f \colon G \rightarrow \C : f/\omega \in C_0(G) \}.$$
Moreover, $L^1(G, \omega)$ is a closed ideal of $M(G, \omega)$. When $\omega(t) = \omega(t^{-1})$ for all $t \in G$ both $M(G, \omega)$ and $L^1(G, \omega)$ are Banach *-algebras, with the involution given by 
$$f^*(t) = \Delta(t^{-1})\overline{f(t^{-1})} \qquad (t \in G),$$
where $\Delta$ is the modular function of $G$.

Given $\pi \in \widehat{G}$ and $\xi \in \mathcal{H}_{\pi}$, we shall write
$$\mathcal{I}_{\pi,\xi} = \{ f \in L^1(G) : \pi(f)\xi = 0 \} \qquad
\text{and}
\qquad \mathcal{J}_{\pi,\xi} = \{ \mu \in M(G) :  \pi(\mu)\xi = 0 \}.$$
Moreover, given a weight $\omega$ on $G$, we shall write $\mathcal{I}_{\pi,\xi, \omega} = \mathcal{I}_{\pi,\xi}\cap L^1(G, \omega)$ and $\mathcal{J}_{\pi,\xi, \omega} = \mathcal{J}_{\pi,\xi} \cap M(G, \omega).$ These are always closed left ideals of $L^1(G, \omega)$ and $M(G, \omega)$ respectively.

Let $A$ be a Banach *-algebra. By a \textit{projection} we mean a self-adjoint idempotent. We say that $p \in A$ is a \textit{minimal projection} if $Ap$ is a minimal left ideal in $A$. When $A$ is a C*-algebra or has the form $L^1(G)$ for some locally compact group $G$, a projection $p \in A$ is minimal if and only if $pAp = \C p$. However, this is in general a stronger condition than being minimal with respect to the usual partial order on projections.

There has long been known a procedure for constructing an irreducible *-representation $\pi$ from a minimal projection $p$ in a Banach *-algebra $A$, summarised by Barnes in \cite{B80}. Another important contribution to this theory is Valette's paper \cite{V84}. Following Barnes, we say that $\pi$ is \textit{determined by $p$}. Barnes \cite[Theorem 1]{B80} proved the following theorem, relating integrable representations with minimal projections.

\begin{theorem}[Barnes]		\label{0.7}
	Let $G$ be a unimodular locally compact group and let $\pi \in \widehat{G}$. Then $\pi$ is integrable if and only if $\pi$ is determined by a minimal projection in $L^1(G)$.
\end{theorem}

 Since minimal projections are defined as those that generate minimal left ideals, Barnes' Theorem tacitly establishes a link between left ideals and integrable representations. In Proposition \ref{4.11} below we shall reformulate this result in terms of maximal left ideals of the measure algebra of the group.

We say that a Banach *-algebra $A$ is \textit{Hermitian} if every self-adjoint element of $A$ has real spectrum. A locally compact group $G$ is called Hermitian if $L^1(G)$ is Hermitian. Examples of Hermitian locally compact groups include compact extensions of nilpotent locally compact groups and Moore groups \cite{P01}. It has recently been shown by Samei and Weirsma \cite{SW} that Hermitian locally compact groups are amenable, solving a long-standing open problem.

Our interest in Hermitian Banach *-algebras is due to Palmer's Theorem \cite{P72}, which states that a Banach *-algebra $A$ is Hermitian if and only if every maximal modular left ideal of $A$ has the form $\{ a \in A : \psi(a^*a) = 0 \}$, for some pure state $\psi$ on $A$. For the special case of a Hermitian group algebra, this theorem states the following.

\begin{theorem}[Palmer]		\label{0.6}
	Let $G$ be a Hermitian locally compact group. The maximal modular left ideals of $L^1(G)$ all have the form $\mathcal{I}_{\pi, \xi},$
	where $\pi \in \widehat{G}$, and $\xi$ is a unit vector in $\mathcal{H}_\pi$.
\end{theorem}

Unfortunately, Palmer's Theorem does not tell us which of the left ideals $\mathcal{I}_{\pi, \xi}$ are maximal modular, and it seems to us that this is generally difficult to determine. Since $L^1(G)$ is semisimple, we know that it has enough maximal modular left ideals to separate points. We shall prove a more precise result for Hermitian CCR locally compact groups in Proposition \ref{4.13}.

We next recall some background on dual Banach algebras. A \textit{dual Banach algebra} is a Banach algebra $A$ that has a Banach space predual $X$ in such a way that the multiplication on $A$ is separately weak*-continuous. For example, the measure algebra of a locally compact group $G$ is a dual Banach algebra with predual $C_0(G)$; more generally, given a weight $\omega$ on $G$, the weighted measure algebra $M(G, \omega)$ is a dual Banach algebra with predual $C_0(G, 1/\omega)$ (as verified in \cite{W0}). This is where we use the assumption that $\omega$ is continuous. Other examples include von Neumann algebras, and the Fourier-Stieltjes algebra $B(G)$ of a locally compact group $G$. 
%In this article we study left ideals of $M(G, \omega)$ which are closed with respect to the weak*-topology induced by $C_0(G, 1/\omega)$. By a weak*-closed maximal left ideal we mean a maximal left ideal closed in this topology; this may be different from a left ideal which is maximal subject to being weak*-closed.

Given a Banach algebra $A$ we shall write $M(A)$ for its multiplier algebra.
 %The multiplier algebra of $L^1(G,\omega)$ may be identified isometrically with $M(G,\omega)$, for any locally compact group $G$ and any weight $\omega$ on $G$.
Recall that $A$ is faithful if $Ax = \{ 0 \}$ implies that $x = 0 \ ( x \in A)$, and also $xA = \{ 0 \}$ implies that $x = 0 \ (x \in A)$. The map canonical $A \rightarrow M(A)$ is injective whenever $A$ is faithful, and we shall typically identify a faithful Banach algebra with its image inside $M(A)$. When $A$ has a contractive approximate identity, the canonical map is an isometry. This is the case for Banach algebras of the form $L^1(G, \omega)$, and in this case the multiplier algebra may be identified isometrically with $M(G, \omega)$ \cite[Theorem 7.14]{DL}. In \cite{Da, HA, W3} multiplier algebras that are also dual Banach algebras were studied, and in \cite{W3} correspondences between the ideal structure of a Banach algebra and that of its multiplier algebra were explored. The present work could be seen as a continuation of these investigations.

\section{Left Ideals of Finite Codimension}
\noindent
In this section we shall prove that, for a locally compact group $G$, the measure algebra $M(G)$ can never have a weak*-closed left ideal of finite codimension unless $G$ is compact, and note some of the immediate consequences of this fact. Our strategy is to prove the corresponding result for the predual, namely that, when $G$ is not compact, $C_0(G)$ has no non-zero finite-dimensional linear subspaces which are invariant under left translation.  

We begin by proving a short lemma.

\begin{lemma} 	\label{1.1}
	Let $n \in \N$, and let $\lambda_0, \ldots, \lambda_n \in \C$, and suppose that 
	$$\vert \lambda_j \vert \leq 2^{-n} \sum_{i \neq j} \vert \lambda_i \vert \quad (j = 0, \ldots, n).$$
	Then $\lambda_0 = \cdots = \lambda_n = 0$.
\end{lemma}

\begin{proof}
Without loss of generality, assume that $|\lambda_0| \geq |\lambda_j| \ (j = 0, \ldots, n)$. Then
$$|\lambda_0| \leq  2^{-n} \sum_{i=1}^n \vert \lambda_i \vert \leq \frac{n}{2^n} |\lambda_0|,$$
forcing $\lambda_0 = 0$, and hence $\lambda_j = 0 \ (j = 1, \ldots, n)$.

	%We proceed by induction on $n$, with the base case $n=0$ being trivial (with empty sums interpreted as zero).
	%Suppose that $n \geq 1$. Fix $j \in \{ 0, \ldots, n \}$. 
	%We may choose $k \in \{ 0, \ldots, n \}$ with $j \neq k$. Our hypothesis implies that, for every $p \neq k$, we have 
	%\begin{align*}
	%\vert \lambda_p \vert &\leq 2^{-n} \sum_{i \neq p} \vert \lambda_i \vert  
	%= 2^{-n} \left( \vert \lambda_k \vert + \sum_{i \neq p, k}\vert \lambda_i \vert \right)
	%\leq 2^{-n} \left( 2^{-n} \sum_{i \neq k } \vert \lambda_i \vert 
	%+ \sum_{i \neq p, k} \vert \lambda_i \vert  \right) \\
	%&= 2^{-2n} \vert \lambda_p \vert + (2^{-n} + 2^{-2n}) \sum_{i \neq p, k} \vert \lambda_i \vert.
	%\end{align*}
	%Rearranging, we obtain
	%$$\vert \lambda_p \vert \leq \frac{2^{-n} + 2^{-2n}}{1-2^{-2n}} \sum_{i \neq p, k} \vert \lambda_i \vert = \frac{2^{-n}}{1-2^{-n}} \sum_{i \neq p, k} \vert \lambda_i \vert \leq 2^{-(n-1)} \sum_{i \neq p, k} \vert \lambda_i \vert$$
	%for every $p \neq k$. 
	%We may now apply the induction hypothesis to conclude that $\lambda_p = 0 $ for every $p \neq k$, so that in particular $\lambda_j = 0.$ As $j$ was arbitrary, this gives the result.
\end{proof}
In the next proposition, given $t \in G$ and $\phi \in C_0(G)$, we write $(L_t \phi )(s) := \phi(ts) \ (s \in G)$.

\begin{proposition}		\label{1.2}
	Let $G$ be a locally compact group. Then $C_0(G)$ contains a non-zero finite-dimensional linear subspace invariant under left translation if and only if $G$ is compact.
\end{proposition}

\begin{proof}
	If $G$ is compact, then $\C 1$ is such a subspace. Suppose that $G$ is not compact, and assume towards a contradiction that there is a left-translation-invariant linear subspace $F \subset C_0(G)$ of dimension $n \in \N$. 
	
	Choose $\phi \in F$ of norm $1$, and let $x_0 \in G$ satisfy $1 = \vert \phi (x_0) \vert$. Let $K \subset G$ be a compact subset such that $\vert \phi (t) \vert < 2^{-n}$ whenever $t \in G \setminus K$. Inductively choose elements $t_0, t_1, \ldots, t_n \in G$ with $t_0 = e$, and such that 
	$$t_j \notin \left( \bigcup_{i = 0}^{j-1} x_0 K^{-1} t_i \right) \cup
	\left( \bigcup_{i = 0}^{j-1} K x_0^{-1}t_i \right)
	% \cup \left( \bigcup_{i=0}^{j-1}t_ix_0K^{-1} \right) 
	\quad (j = 1, \ldots, n).$$
	This is possible since each such set is compact, whereas $G$ is not. Our sequence $(t_j)$ satisfies 
	$$\vert \phi(t_it_j^{-1} x_0) \vert < 2^{-n} \quad (i \neq j, \ i,j = 0, \ldots, n).$$
	
	Since each function $L_{t_i} \phi $ also belongs to $F$, which has dimension $n$, there must exist $\lambda_0, \ldots, \lambda_n \in \C$, not all zero, such that 
	$$\lambda_0 L_{t_0}\phi + \cdots + \lambda_n L_{t_n} \phi = 0.$$
	%By scaling, we may assume that there is a $k \in \{0, \ldots, n \}$ such that $\lambda_k = 1$.
	Given $j \in \{0, \ldots, n \}$ we have 
	\begin{align*}
	\vert \lambda_j \vert &= \vert \lambda_j \phi(x_0) \vert 
	= \vert \lambda_j (L_{t_j} \phi)(t_j^{-1}x_0) \vert = 
	\left\vert \sum_{i \neq j} \lambda_i (L_{t_i}\phi)(t_j^{-1}x_0) \right \vert 
	= \left\vert \sum_{i \neq j} \lambda_i \phi(t_it_j^{-1}x_0) \right\vert \leq 2^{-n}\sum_{i \neq j} \vert \lambda_i \vert,
	\end{align*}
	which implies that $\lambda_j = 0$ for all $j\in \{0, \ldots, n \}$ by Lemma \ref{1.1}. This contradiction completes the proof.
\end{proof} 

\begin{corollary}		\label{1.3}
	Let $G$ be a locally compact group which is not compact. Then $M(G)$ contains no proper weak*-closed left ideals of finite codimension, and hence no finitely-generated, closed left ideals of finite codimension.
\end{corollary}

\begin{proof}
	Assume towards a contradiction that $M(G)$ has a proper weak*-closed left ideal $J$ of finite codimension. By \cite[Lemma 3.3]{W1}, $J_\perp$ is a non-zero closed linear subspace of $C_0(G)$, invariant under left translation, and since $(J_\perp)' \cong \frac{M(G)}{J}$, it is finite-dimensional. This contradicts Proposition \ref{1.2}.
\end{proof}

One consequence of this corollary is that $M(G)$ has no weak*-closed maximal left ideals when $G$ is a non-compact Moore 
(see Theorem \ref{0.3}(i), proved in Section 7 below). Another consequence is the following result, which puts a structural constraint on those locally compact groups $G$ for which $M(G)$ has a weak*-closed maximal left ideal. 

\begin{proposition}		\label{1.4}
	Let $G$ be a locally compact group such that $M(G)$ has a weak*-closed maximal left ideal. Then $G$ has compact centre.
\end{proposition}

\begin{proof}
	Let $J$ be a weak*-closed maximal left ideal of $M(G)$, and consider $I = J \cap M(Z(G))$. Then $I$ is closed in the weak*-topology on 
	$M(Z(G))$. Let $T \colon M(G) \rightarrow \B(M(G)/J)$ be given by $T(\mu)(\nu + J ) = \mu*\nu + J \ (\mu, \nu \in M(G)).$ 
	Then since $T[M(Z(G))] \subset Z(T[M(G)])$, and $M(G)/J$ is a simple left $M(G)$-module, we must have $T[M(Z(G))] = \C \id_{M(G)/J}$ by
	\cite[Theorem 4.2.11]{P94}. Hence $\ker T|_{M(Z(G))}$ has codimension-1 in $M(Z(G))$, but also clearly 
	$\ker T|_{M(Z(G))} \subset I$, so that $I$ has codimension-1. It now follows from Corollary \ref{1.3} that $Z(G)$ is compact, as required.
\end{proof}

We shall show in Example \ref{eg4.13b} below that there are locally compact groups $G$ for which $M(G)$ has no weak*-closed maximal left ideals, but for which $G$ has trivial centre. In Corollary \ref{4.9c} we shall show that Proposition \ref{1.4} can be strengthened in the case of a Hermitian group.

\section{Classification Results for Weak*-Closed Maximal Left Ideals}
\noindent
In this section we shall give a classification of the weak*-closed maximal left ideals of $M(G)$, valid for $G$ belonging to a large family of Hermitian locally compact groups. Our classification shall be in terms of the irreducible representations of $G$ and the maximal modular left ideals of $L^1(G)$. The latter objects seem very difficult to determine if the group has infinite dimensional irreducible representations, even when the representation theory is well understood. But, as we shall see in subsequent sections, our classification is nonetheless good enough to give us valuable insights into the ideal structure of $M(G)$. As we shall detail in Section 7, for many examples our classification will actually show that the measure algebra has no weak*-closed maximal left ideals, and hence no finitely-generated maximal left ideals either.

We begin by proving some lemmas in a very general context. We hope that this will allow us to prove classification results about weak*-closed left ideals of other multiplier algebras which are also dual Banach algebras, such as the Fourier--Stieltjes algebra of a locally compact amenable group, in future work. We shall make use of these general lemmas in this article to study the ideal structure of measure algebras and their weighted analogues.

The first lemma tells us how to determine maximality of a left ideal of $M(A)$ by looking at a left ideal of $A$, in the special case that $M(A)$ is a dual Banach algebra and $A$ is faithful and weak*-dense in $M(A)$.
Note that, by \cite[Lemma 5.2]{W3}, $A$ is weak*-dense in $M(A)$ whenever $M(A)$ is a dual Banach algebra and $A$ has a bounded approximate identity.

\begin{lemma} 		\label{4.3a}
	Let $B$ be a dual Banach algebra, and let $A$ be a weak*-dense ideal of $B$.
	\begin{enumerate}
		\item[ \rm (i)] If $J$ is a weak*-closed left ideal of $B$ such that $J \cap A$ is a maximal modular left ideal of $A$ then 
		$J$ is maximal.
		\item[ \rm (ii)] If $I$ is a maximal modular left ideal of $A$, then either $\overline{I}^{w^*}$ is a maximal left ideal in $B$, or else $I$ is weak*-dense in $B$.
	\end{enumerate}
\end{lemma}

\begin{proof}
	(i) Set $I = J \cap A$, and let $u \in A$ be a modular right identity for $I$.  Let $b \in B$. By weak*-density of $A$ in $B$ we may take a net 
	$(a_\alpha)$  in $A$ which converges to $b$ in the weak*-topology. Then as $B$ is a dual Banach algebra $a_\alpha u$ weak*-converges to $bu$. Since $u $ is a modular right identity for $I$, for all $\alpha$ we have $a_\alpha - a_\alpha u \in I \subset J$, and hence 
	$$\lim_{w^*, \, \alpha} a_\alpha - a_\alpha u = b - b u \in J$$
	by weak*-closedness.
	Hence $b = (b - b u) + b u \in J + A$. As $b$ was arbitrary we have $J + A = B$.
	
	Observe that as left $B$-modules we have
	\begin{equation}		\label{eq4.1}
	\frac{B}{J} = \frac{J + A}{J} \cong  \frac{A}{J \cap A} = \frac{A}{I}.
	\end{equation}
	Since $A/I$ is a simple left $A$-module, it is also a simple left $B$-module, and hence $B/J$ is simple by the isomorphism. It follows that $J$ is a maximal left ideal.
	
	(ii) We shall suppose that $I$ is not weak*-dense in $B$, and show that $J := \overline{I}^{w^*}$ must be maximal. Indeed, in this case we cannot have $J \supset A$, and hence $J\cap A$ is a proper left ideal of $A$ containing $I$, which is equal to $I$ by maximality. Hence the result follows from (i).
\end{proof}

\begin{corollary}		\label{4.3g}
	Let $G$ be a locally compact group, and let $\omega$ be a  weight on $G$.
	Let $\pi \in \widehat{G}$ and let $\xi \in S_{\mathcal{H}_\pi}$, and suppose that $\mathcal{J}_{\pi, \xi, \omega}$ is weak*-closed. If $\mathcal{I}_{\pi, \xi, \omega}$ is a maximal modular left ideal of $L^1(G, \omega)$, then
	 $\mathcal{J}_{\pi, \xi}$ is a maximal left ideal of $M(G)$.
\end{corollary}

\begin{proof}
	Since $\mathcal{I}_{\pi, \xi, \omega} = \mathcal{J}_{\pi, \xi, \omega} \cap L^1(G, \omega)$, this follows from Lemma \ref{4.3a}(i).
\end{proof}
 
 Our next lemma complements Lemma \ref{4.3a}, and tells us how weak*-closed left ideals in $M(A)$ give us maximal left ideals in $A$ itself.

\begin{lemma}		\label{4.3b}
	Let $A$ be a Banach algebra with a bounded approximate identity, such that $M(A)$ is a dual Banach algebra. Let $J$ be a maximal left ideal of $M(A)$ which is weak*-closed. Then $J\cap A$ is a maximal modular left ideal in $A$. 
\end{lemma}

\begin{proof}
	 Set $I = J \cap A$, and note that $A$ is weak*-dense in $M(A)$ by \cite[Lemma 5.2]{W3}. Since $J$ is proper and weak*-closed it does not contain $A$, so that $I$ is proper. Also we have $J  + A = M(A)$ by maximality.
	Hence we may take $u \in A$ satisfying $1 - u \in J$, and for all $a \in A$ it then follows that $a - au \in I$, so that $u$ is a modular right identity for $I$. As such $I$ is a proper modular left ideal of $A$. Note that we cannot have $I= \{0\}$, since this would mean that $J$ annihilates $A$, contradicting the fact that it contains non-zero multipliers.
	
	Observe that the isomorphism \eqref{eq4.1} again holds, this time with $B = M(A)$.
	Since $M(A)/J$ is a simple left $M(A)$-module, so is $A/I$. 
	To see that $A/I$ is in fact a simple left $A$-module, consider a maximal modular left ideal $K$ of $A$ containing $I$, and note that $K$ is also a left ideal in $M(A)$ since $A$ has an approximate identity. This forces $K = I$ by $M(A)$-simplicity of $A/I$. Hence $I$ is maximal modular, completing the proof.
	\end{proof}

\begin{corollary}		\label{4.3c}
	Let $G$ be a locally compact Hermitian group, and let $J$ be a weak*-closed maximal left 
	ideal of $M(G)$. Then there exists $\pi \in \widehat{G}$ and $\xi \in S_{\mathcal{H}_\pi}$ such that $J = \mathcal{J}_{\pi, \xi}$.
\end{corollary}

\begin{proof}
	By Lemma \ref{4.3b}, $J \cap L^1(G)$ is a maximal modular left ideal in $L^1(G)$, and hence has the form $\mathcal{I}_{\pi, \xi}$, for some  $\pi \in \widehat{G}$ and $\xi \in S_{\mathcal{H}_\pi}$ by Palmer's Theorem (Theorem \ref{0.6}). Let $(e_\alpha) \subset L^1(G)$ be a net that weak*-converges to $\delta_e$. Then for all $\mu \in \mathcal{J}_{\pi,\xi}$ we have $(e_\alpha*\mu) \subset L^1(G)\cap \mathcal{J}_{\pi, \xi}$, and also $\lim_{\alpha}e_\alpha *\mu = \mu$ in the weak*-topology.
	Therefore $\mathcal{J}_{\pi, \xi} \subset \overline{L^1(G)\cap \mathcal{J}_{\pi, \xi}}^{w^*}$. The same argument shows that $J \subset \overline{L^1(G)\cap J}^{w^*},$ from which it follows that $J = \overline{L^1(G)\cap J}^{w^*}$ because $J$ is weak*-closed. Hence 
	$$\mathcal{J}_{\pi, \xi} \subset \overline{L^1(G)\cap \mathcal{J}_{\pi, \xi}}^{w^*}
	= \overline{\mathcal{I}_{\pi, \xi}}^{w^*} =  \overline{L^1(G) \cap J}^{w^*} = J.$$
	By Lemma \ref{4.3a}(ii) $\mathcal{J}_{\pi, \xi}$ is a maximal left ideal in $M(G)$, and so $J = \mathcal{J}_{\pi, \xi}$, as required.
\end{proof}

The next lemma has consequences both for the weighted and unweighted case. Given a Hilbert space $\mathcal{H}$, we shall denote the conjugate Hilbert space by $\overline{\mathcal{H}}$.  Also, for a dual Banach space $E$ with predual $X$, we shall write $\langle x, \lambda \rangle$ as $\langle x, \lambda \rangle_{(X, \, E)}$, for $x \in X$ and $\lambda \in E$, when we wish to clarify the exact dual pairing.

\begin{lemma}		\label{4.3}
	Let $G$ be a locally compact group, and let $\omega$ be a  weight on $G$. Let $(\pi, \mathcal{H}_\pi)$ be a representation of $G$, and suppose either that $\pi$ vanishes at infinity, or that $1/\omega \in C_0(G)$. 
	Then for each $\xi \in \mathcal{H}_\pi$ the map
	$$\theta_\xi \colon M(G, \omega) \rightarrow \mathcal{H}_\pi, \qquad \mu \mapsto \pi(\mu) \xi$$
	is weak*-weak* continuous, with preadjoint given by
	$$\varphi_\xi \colon \overline{\mathcal{H}}_\pi \rightarrow C_0(G, 1/\omega), \qquad \eta \mapsto \overline{\xi *_\pi \eta}.$$
\end{lemma}

\begin{proof}
	If $\pi$ vanishes at infinity then the map $\varphi_\xi$ is well-defined since in this case 
	$\xi*_\pi \eta \in C_0(G) \subset C_0(G, 1/\omega)$ for all $\eta \in \overline{\mathcal{H}}_\pi$. Similarly, for general $\pi$ the condition 
	$1/\omega \in C_0(G)$ ensures that $\varphi_\xi$ is well-defined, since $\xi*_\pi \eta$ is always bounded.
	In either case $\varphi_\xi$ is a bounded linear map with $\Vert \varphi_\xi \Vert \leq \Vert \xi \Vert$. Moreover, for every 
	$\eta \in \overline{\mathcal{H}}_\pi$, and every $\mu \in M(G)$, we have
	\begin{align*}
	\langle \eta, \theta_\xi(\mu) \rangle_{(\overline{\mathcal{H}}_\pi, \, \mathcal{H}_\pi)} &= \langle \eta, \pi(\mu) \xi \rangle = \overline{\langle \pi(\mu) \xi, \eta \rangle}
	= \int_G \overline{\langle \pi(t) \xi, \eta \rangle} \dd \mu(t) \\
	&= \langle \overline{\xi *_\pi \eta}, \mu \rangle_{(C_0(G, 1/\omega), \, M(G, \omega))}
	= \langle \eta, \varphi_\xi '(\mu) \rangle_{(\overline{\mathcal{H}}_\pi, \, \mathcal{H}_\pi)}.
	\end{align*}
	Hence $\theta_\xi = \varphi_\xi '$.
\end{proof}

\begin{corollary}		\label{4.4}
	Let $G$ be a locally compact group, let $(\pi, \mathcal{H}_\pi)$ be a representation of $G$ which vanishes at infinity, and let $\xi \in \mathcal{H}_\pi$. 
	\begin{enumerate}
		\item[\rm (i)] The left ideal of $M(G)$ given by 
		$\mathcal{J}_{\pi, \xi}$ is weak*-closed.
		\item[\rm (ii)] If $\mathcal{I}_{\pi,\xi}$ is a maximal modular left ideal of $L^1(G)$, then $\mathcal{J}_{\pi, \xi}$  is a weak*-closed maximal left ideal of $M(G)$.
	\end{enumerate}
\end{corollary}

\begin{proof}
	In the notation of Lemma \ref{4.3}, $\mathcal{J}_{\pi, \xi} = \ker \theta_\xi$, and part (i) now follows from that lemma.
	Part (ii) then follows from Corollary \ref{4.3g}.
\end{proof}

We record the following folklore result for the convenience of the reader. 
%\textbf{[If we are going to try for a shorter paper in a very good journal, we should take this out].}

\begin{lemma}[folklore]		\label{4.9a}
	Let $\mathcal{M}$ be a commutative von Neumann algebra which admits an ergodic action by automorphisms by a finite group $G$. Then $\mathcal{M}$ is finite dimensional.
\end{lemma}

\begin{proof}
	Assume towards a contradiction that $\mathcal{M}$ is infinite dimensional, and let $u_1,u_2, \ldots$ be an infinite linearly-independent subset of $\mathcal{M}$. Let $\mathcal{N}$ be the von Neumann algebra generated by $\{ gu_i : g \in G, i \in \N \}$. Then $G$ also acts ergodically on $\mathcal{N}$, which is separable, and as such may be identified with $L^\infty(X)$, for $X$ one of the following spaces with the obvious measure: $[0, 1], \N, [0,1]\cup \N$, or $\{1, \ldots, n \}, [0,1]\cup \{ 1, \ldots, n \}$ for some $n \in \N$. The action of $G$ on $\mathcal{N}$ corresponds to an ergodic action of $G$ on the underlying measure space $X$. However, there is no such action of a finite group on any of the infinite sets listed above, and hence
	$\mathcal{N} = L^\infty \{ 1, \ldots, n \}$ for some $n \in \N$. In particular $\mathcal{N}$ is finite-dimensional, contradicting the choice of $\{u_1, u_2, \ldots \}$. 
\end{proof}

The following proposition is the remaining key ingredient in the proof of our main classification theorem (Theorem \ref{4.4a}). It will also be an important tool in proving Theorem \ref{0.3}(iii).

\begin{proposition}		\label{4.9b}
	Let $G$ be a locally compact group, and let $\pi \in \widehat{G}$. Suppose that $G$ has a finite-index, closed normal subgroup $N$ such that $\pi^{-1}(Z(\pi(N)))$ is not compact. Then for every $\xi \in S_{\mathcal{H}_{\pi}}$ the left ideal $\mathcal{J}_{\pi,\xi}$ is weak*-dense in $M(G)$.
\end{proposition}

\begin{proof}
	Suppose towards a contradiction that there exists $\xi \in S_{\mathcal{H}_\pi}$ such that $\mathcal{J}_{\pi, \xi}$ is weak*-closed. 
	Consider $\mathcal{M}$, the von Neumann algebra generated by $\pi(Z(N))$. The action of $G$ on $Z(N)$ by conjugation induces an action of $G/N$ on $\mathcal{M}$ by automorphisms, and the fixed point set for this action consists of operators which commute with 
	$\pi(G)$, and hence consists of scalars. In other words, $G/N$ acts ergodically on $\mathcal{M}$, forcing $\mathcal{M}$ to be finite dimensional by Lemma \ref{4.9a}.

	Let $H = \pi^{-1}(Z(\pi(N)))$, which is a closed subgroup of $G$. Set $\mathcal{K} = \mathcal{M} \xi$, which is a finite-dimensional linear subspace of $\mathcal{H}_\pi$, and define $\rho \colon H \rightarrow \B(\mathcal{K})$ to be 
	$$\rho(z)\eta = \pi(z)\eta \quad (z \in H, \ \eta \in \mathcal{K}).$$
	Since $\rho$ is a finite-dimensional representation of $H$, the left ideal $\mathcal{J}_{\rho, \xi}$
	of $M(H)$ has finite codimension, and hence must be weak*-dense by Corollary \ref{1.3}. Moreover,
	we have $\mathcal{J}_{\pi, \xi} \cap M(H) = \mathcal{J}_{\rho, \xi}$. 
	
	Let $(\mu_\alpha) \subset \mathcal{J}_{\rho, \xi}$ be a net converging in the weak*-topology on $M(H)$ to $\delta_e$. Given $f \in C_0(G)$, we have
	$$\int_G f \dd \mu_\alpha = \int_H f \dd \mu_\alpha \rightarrow f(e),$$
	since each $\mu_\alpha$ is supported on $H$, and $f \vert_H \in C_0(H)$. Hence $\mu_\alpha$ converges to $\delta_e$ with respect to the weak*-topology on $M(G)$ as well. We have $\delta_e \in \overline{\mathcal{J}_{\pi, \xi}}^{w^*}$, which is a left ideal in $M(G)$, and hence 
	$\overline{\mathcal{J}_{\pi, \xi}}^{w^*} = M(G)$, as required.	
\end{proof}

\begin{corollary}		\label{4.9c}
	Let $G$ be a locally compact group.
	\begin{enumerate}
		\item[\rm (i)] If $\pi \in \widehat{G}$ has the property that $\pkr \pi$ is not compact, then $\mathcal{J}_{\pi, \xi}$ is weak*-dense in $M(G)$ for every $\xi \in S_{\mathcal{H}_\pi}$.
		
		\item[\rm (ii)] If $G$ is Hermitian and has a finite index closed normal subgroup $N$ such that $Z(N)$ is non-compact, then $M(G)$ has no weak*-closed, maximal left ideals. \qed
	\end{enumerate}
\end{corollary}

The next lemma gives a converse to Corollary \ref{4.4} for a particular class of groups.
%\textbf{[I could merge this lemma into the proof of the theorem.]}

\begin{lemma}		\label{4.12}
	Let $G$ be a locally compact group for which every irreducible representation vanishes at infinity
	modulo its projective kernel. Let $\pi \in \widehat{G}$, and let $\xi \in S_{\mathcal{H}_\pi}$ 
	be such that $\mathcal{J}_{\pi, \xi}$ is weak*-closed. Then $\pi$ vanishes at infinity.
\end{lemma}

\begin{proof}
	By Corollary \ref{4.9c}(i) $\pkr \pi$ must be compact, since otherwise
	$\mathcal{J}_{\pi, \xi}$ would be weak*-dense. 
	Since by hypothesis $\pi$ vanishes at infinity modulo $\pkr \pi$, it follows that $\pi$ itself vanishes at infinity.
\end{proof}

We can now prove our main classification theorem.

\begin{theorem}		\label{4.4a}
	Let $G$ be a Hermitian locally compact group with the property that every $\pi \in \widehat{G}$ vanishes at infinity modulo its projective kernel. 
	Then the weak*-closed maximal left ideals of $M(G)$ are given by $\mathcal{J}_{\pi, \xi}$,
	where $\pi \in \widehat{G}$ vanishes at infinity, and $\xi$ is a unit vector in $\mathcal{H}_\pi$ such that $\mathcal{I}_{\pi, \xi}$ is a maximal modular left ideal in $L^1(G)$.
\end{theorem}

\begin{proof}
	Let $\pi \in \widehat{G}$ vanish at infinity, and let $\xi \in L^1(G)$ be such that $\mathcal{I}_{\pi, \xi}$ is a maximal modular left ideal in $L^1(G)$. Then $\mathcal{J}_{\pi, \xi}$ is a weak*-closed maximal left ideal by Corollary \ref{4.4}. Conversely, if we take a weak*-closed maximal left ideal $J$ of $M(G)$, then Corollary \ref{4.3c} implies that $J = \mathcal{J}_{\pi, \xi}$ for some $\pi \in \widehat{G}$ and $\xi \in S_{\mathcal{H}_\pi}$. Now $\pi$ must vanish at infinity by Lemma \ref{4.12}.
\end{proof}

\textit{Remark.} By the remarks on page 221 of \cite{BT82} every irreducible representation of a connected nilpotent Lie group vanishes at infinity modulo its kernel, and as such the previous theorem applies. By \cite{HM79} it also applies to Hermitian, connected, real or complex algebraic groups. The theorem can also be seen to apply to the Fell groups $\Q_p \rtimes \mathcal{O}_p \ (p \in \N \text{ prime}),$ and $\R^n \rtimes SO(n) \ (n \in \N, \ n \geq 2)$, since the representations of these groups are classified (see for instance \cite[Examples 4.42-4.43]{KT}).  
\vskip 2mm

We conclude this section with a classification result for certain weighted measure algebras. In the case of a compact group $G$ there is a bijective correspondence between the maximal modular left ideals of $L^1(G)$ and the weak*-closed maximal left ideals of $M(G)$ \cite[Theorem 6.2]{W3}. The next proposition shows that, for certain weights on certain non-compact groups $G$ we get an analogous correspondence for $L^1(G, \omega)$ and $M(G, \omega)$. In particular, we get a description of the weak*-closed maximal left ideals of $M(G, \omega)$ which does not depend on the asymptotic properties of the representations of $G$.

\begin{theorem}		\label{4.3f}
	Let $G$ be a compactly-generated, locally compact group of polynomial growth. Let $\omega$ be a weight on $G$ such that
	$1/\omega \in C_0(G)$, $\omega(t^{-1}) = \omega(t)$, and $\lim_{n \rightarrow \infty} \omega(t^n)^{1/n} = 1 \ (t \in G)$. 
	%Let $\pi \in \widehat{G}, \xi \in S_{\mathcal{H}_\pi}$. 
	Then every maximal modular left ideal of $L^1(G, \omega)$ has the form $\mathcal{I}_{\pi, \xi, \omega}$ for some $\pi \in \widehat{G}$, and $\xi \in S_{\mathcal{H}_\pi}$.
	Moreover, the weak*-closed maximal left ideals of $M(G, \omega)$ are exactly given by $\mathcal{J}_{\pi, \xi, \omega}$, where $\pi \in \widehat{G}$, and $\xi \in S_{\mathcal{H}_\pi}$ is such that $\mathcal{I}_{\pi, \xi, \omega}$ is maximal modular in $L^1(G, \omega)$. 
\end{theorem}

\begin{proof}
	Let $\gamma$ denote the maximal C*-norm of $L^1(G)$, and $\gamma_\omega$ denote that of 
	$L^1(G, \omega)$. Also let $r$ and $r_\omega$ denote the spectral radius in $L^1(G)$ and $L^1(G, \omega)$ respectively. By \cite[Theorem 1.3]{FGL} both $L^1(G)$ and $L^1(G, \omega)$ are Hermitian so that for
	$f \in L^1(G, \omega)$ we have 
	$$\gamma_\omega(f) = r_\omega(f^* *f)^{1/2} =r(f^* *f)^{1/2} = \gamma(f),$$
	where we have used condition (iv) of \cite[Theorem 1.3]{FGL} to get the second equality. Therefore the universal C*-algebra of $L^1(G, \omega)$ is $C^*(G)$. It follows that the pure states on $L^1(G, \omega)$ all have the form $f \mapsto \langle \pi(f)\xi, \xi \rangle \ (f \in L^1(G, \omega))$, for some $\pi \in \widehat{G}$ and $\xi \in S_{\mathcal{H}_\pi}$, and hence the maximal modular left ideals have the required form by Palmer's Theorem \cite{P72}.
	
	Now let $J$ be a weak*-closed maximal left ideal of $M(G, \omega)$. Then by Lemma \ref{4.3b} $J \cap L^1(G, \omega)$ is a maximal modular left ideal of $L^1(G, \omega)$, and hence has the form 
	$\mathcal{I}_{\pi, \xi, \omega}$, for some $\pi \in \widehat{G}$, and $\xi \in S_{\mathcal{H}_\pi}$, by the previous paragraph.
	We may now proceed exactly as in Corollary \ref{4.3c} to see that $J = \mathcal{J}_{\pi, \xi, \omega}$.
	
	Finally consider  $\pi \in \widehat{G}$, and $\xi \in S_{\mathcal{H}_\pi}$ is such that $\mathcal{I}_{\pi, \xi, \omega}$ is maximal modular in $L^1(G, \omega)$.
	By Lemma \ref{4.3} each ideal $\mathcal{J}_{\pi,\xi,\omega}$ is weak*-closed, and hence maximal by Lemma \ref{4.3a}(i).
\end{proof}

\textit{Remark.} There are many examples of weights satisfying the hypothesis of the previous theorem. For instance, given $n \in \N$ and $\alpha>0$, define a weight $\omega$ on $\R^n$ by $\omega(x) = (1+\|x\|)^\alpha$. 
\vskip 2mm

\section{Barnes' Theorem}
\noindent
The purpose of this section is to prove an analogue of Barnes' Theorem \ref{0.7} for the setting of certain Hermitian locally compact groups, and representations vanishing at infinity. We shall first give a reformulation of Barnes' Theorem (Proposition \ref{4.11}) that permits a direct comparison with our new result. We begin with a short lemma.

%The purpose of this short section is to give a reformulation of Barnes' theorem \cite{B80} about the correspondence between minimal idempotents and asymptotic properties of representations, that permits a direct comparison with our theorem about representations vanishing at infinity and weak*-closed left ideals of measure algebras later on.

\begin{lemma} 	\label{4.11a}
	Let $A$ be a semisimple faithful Banach algebra. Then any minimal left ideal of $M(A)$ which is contained in $A$ is also minimal in $A$.
\end{lemma}

\begin{proof}
	Let $J$ be such a minimal left ideal of $M(A)$ contained in $A$. If $J$ is not minimal in $A$ then it contains a smaller non-zero left ideal $I$. Since $A$ is semisimple, there must exist $x \in I$ for which $x^2 \neq 0$, and then in particular
	$Ax \neq \{ 0 \}$. Since $Ax$ is a left ideal of $M(A)$, we must have $J = Ax$ by minimality. But $x \in I$ so that $J \subset I$, which is a contradiction.
\end{proof}

%Given an irreducible *-representation $\pi$ of a C*-algebra $A$, it is well known that $\pi$ admits a unique extension to an irreducible *-representation of the multiplier algebra $M(A)$, which we continue to denote by $\pi$. In this section we shall write $\mathcal{K}_{\pi, \xi}$ for the left ideal of $M(C^*_r(G))$ given by $\{ a \in M(C^*_r(G)) : \pi(a)\xi = 0 \}$. 

The following proposition is essentially a rephrasing of Barnes' Theorem in terms of the maximal left ideals of $M(G)$.

\begin{proposition}		\label{4.11}
		Let $G$ be a unimodular locally compact group, and let $\pi \in \widehat{G}$. There exists $\xi \in \mathcal{H}_\pi$ such that $\mathcal{J}_{\pi, \xi}$ is a maximal left ideal generated by a projection if and only if $\pi$ is integrable.
\end{proposition}

\begin{proof}
	Suppose that $\pi$ is integrable. Then $\pi$ is determined by some minimal projection $p$ in $L^1(G)$ as in \cite{B80}. It is clear from the construction of representations determined by a minimal projection that $\pi(p)$ is a rank-one projection, onto a vector that we shall call $\xi$. 
	Then $\delta_e - p \in \mathcal{J}_{\pi, \xi}$. Moreover, for every $\mu \in \mathcal{J}_{\pi, \xi}$ we have $\pi(\mu*(\delta_e - p)) \xi = 0 = \pi(\mu)\xi$ and $\pi(\mu*(\delta_e - p))\eta = \pi(\mu)\eta$ for $\eta \in \{ \xi \}^\perp$, so that $\pi(\mu*(\delta_e - p)) = \pi(\mu)$.
	
	By \cite[Proposition 1]{B80} $\sigma(p) = 0$ for all irreducible representations $\sigma \neq \pi$. Hence $\sigma(\delta_e - p) = \id_{\mathcal{H}_\sigma}$ so that $\sigma(\mu) = \sigma(\mu*(\delta_e - p))$ for all $\mu \in \mathcal{J}_{\pi, \xi}$. Since irreducible representations separate the points of $M(G)$, this forces $\mu = \mu*(\delta_e - p)$ for $\mu \in \mathcal{J}_{\pi, \xi}$, and hence $\mathcal{J}_{\pi, \xi} = M(G)*(\delta_e - p)$.
	
	Now suppose instead that there exists $\xi \in S_{\mathcal{H}_\pi}$ such that  $\mathcal{J}_{\pi, \xi}$ is maximal and that $\mathcal{J}_{\pi, \xi} = M(G)*q$ for some projection $q \in M(G)$. Then by routine algebra $M(G)*p$ is a minimal left ideal in $M(G)$, where $p : = \delta_e - q$. Note that $L^1(G)*p \neq 0$, and hence $L^1(G)*p = M(G)*p$ by minimality, whence $p \in L^1(G)$. It follows from Lemma \ref{4.11a} that $L^1(G)*p$ is a minimal left ideal in $L^1(G)$.
	
	We may now apply Barnes' Theorem \ref{0.7} to conclude that $p$ determines an irreducible integrable representation $\sigma \in \widehat{G}$. Now note that, since $\delta_e - p = q \in \mathcal{J}_{\pi, \xi}$, we have $\pi(p)\xi = \id_{\mathcal{H}_\pi}\xi - \pi(q) \xi = \xi \neq 0$, so that by \cite[Proposition 1]{B80} $\pi = \sigma$.
\end{proof}

Although Palmer's Theorem tells us that, for a Hermitian locally compact group $G$, every maximal modular left ideal of $L^1(G)$ has the form $\mathcal{I}_{\pi, \xi}$, for some $\pi \in \widehat{G}$, and some $\xi \in S_{\mathcal{H}_\pi}$, it seems very difficult to determine, for any non-Moore group $G$, for which $\pi \in \widehat{G}$ and $\xi \in S_{\mathcal{H}_\pi}$ the left ideal $\mathcal{I}_{\pi, \xi}$ is maximal modular. (On the other hand, in the C*-algebraic setting all such left ideals are maximal modular).
Indeed, the only general result that we have been able to establish is the following modest proposition.

\begin{proposition}		\label{4.13}
	Let $G$ be a Hermitian CCR locally compact group. For every $\pi \in \widehat{G}$ there exists  
	$\xi \in S_{\mathcal{H}_\pi}$ such that $\mathcal{I}_{\pi, \xi}$ is a maximal modular left ideal.
\end{proposition}

\begin{proof}
	Fix $\pi \in \widehat{G}$, set $\mathcal{A}_\pi = \pi(L^1(G))$, which is a Banach *-algebra with the quotient topology. We may suppose that $\dim \mathcal{A}_\pi >1$, since otherwise the result is trivial. We shall show that every maximal modular left ideal of $\mathcal{A}_\pi$ lifts to a maximal modular left ideal of $L^1(G)$ of the form $\mathcal{I}_{\pi, \xi}$. Since $\mathcal{A}_\pi$ is semisimple and not equal to $\C$, it has at least one maximal modular left ideal, and this will imply the result.
	
	Let $I$ be a maximal modular left ideal of $\mathcal{A}_\pi$.  Since $L^1(G)$ is Hermitian, so is $\mathcal{A}_\pi$ by \cite{Lep76}, and therefore it follows from Palmer's Theorem \cite{P72} that 
	$I$ is of the form $\{ a \in \mathcal{A}_\pi : \sigma(a) \zeta = 0 \},$ for some irreducible *-representation $\sigma \colon \mathcal{A}_\pi \rightarrow \B(\mathcal{H}_\sigma),$ and some $\zeta \in S_{\mathcal{H}_\sigma}.$ Defining $J = \{ f \in L^1(G) : \pi(f) \in I \}$ gives a maximal modular left ideal of $L^1(G)$.
	
	The *-representation $\sigma$ extends via $\pi$ to $L^1(G)$, and hence to an irreducible *-representation of $C^*(G)$, which we shall denote by $\widetilde{\sigma}$.  With this notation we have $J = \{ f \in L^1(G) : \widetilde{\sigma}(f)\zeta = 0 \}.$ Note that, since $\widetilde{\sigma}$ factors through $\pi$, we have $\ker \widetilde{\sigma} \supset \ker \pi$ in $C^*(G)$.
	In terms of the Fell topology this says that $\widetilde{\sigma} \in \overline{\{ \pi \}}$, which forces $\widetilde{\sigma}$ and  $\pi$ to be equal as elements of $\widehat{G}$, since $\widehat{G}$ is $T_1$ by \cite[Theorem (7.7)]{F} (as $G$ is CCR). It follows that $J$ can be rewritten as $J = \{ f \in L^1(G) : \pi(f) \xi = 0 \}$, for some $\xi \in S_{\mathcal{H}_\pi}$.
\end{proof} 

Next we present our analogue of Barnes' Theorem for irreducible representations vanishing at infinity.

\begin{theorem}		\label{4.15}
	Let $G$ be a Hermitian CCR locally compact group for which every irreducible representation vanishes at infinity modulo its projective kernel, and let $\pi \in \widehat{G}$. Then there exists $\xi \in S_{\mathcal{H}_\pi}$ for which $\mathcal{J}_{\pi, \xi}$ is a weak*-closed maximal left ideal of $M(G)$ if and only if $\pi$ vanishes at infinity.	
\end{theorem}

\begin{proof}
	First suppose that $\pi$ vanishes at infinity. By Proposition \ref{4.13} there exists $\xi \in S_{\mathcal{H}_\pi}$ such that $\mathcal{I}_{\pi, \xi}$ is a maximal modular left ideal of $L^1(G)$. By Theorem \ref{4.4a} $\mathcal{J}_{\pi, \xi}$ is maximal and weak*-closed.
	
	Now suppose instead that there exists $\xi \in S_{\mathcal{H}_\pi}$ such that $\mathcal{J}_{\pi, \xi}$ is a weak*-closed maximal left ideal of $M(G)$. Then Theorem \ref{4.4a} implies that $\pi$ vanishes at infinity.
\end{proof}

\textit{Remark.} Examples of groups satisfying the conditions of the previous theorem include connected nilpotent Lie groups and  $\R^n \rtimes SO(n) \ (n \in \N, \ n \geq 2)$.
\vskip 2mm

%\textit{Remark.}
%A theorem of Dixmier \cite{} \textbf{[cite, or maybe this is obvious by facts in Dixmier, rather than a theorem of his]} says that for a unimodular locally compact group $G$, every irreducible representation that is integrable also vanishes at infinity. The theory that we have developed gives a rather pleasing way to see this for the more restricted class of groups to which it applies which are also unimodular. Indeed, for such a group every integrable representation corresponds to a maximal left ideal of $M(G)$ which is generated by a projection. Being finitely-generated, this left ideal is also weak*-closed, and hence the representation to which it corresponds vanishes at infinity. \textbf{[Is this even true?]}
%\vskip 2mm

\section{Examples}
\noindent
In this section we use the theory that we have developed to give some examples of measure algebras for which the weak*-closed maximal left ideals can be described.
Our first example provides a class of groups whose measure algebras possess an abundance of weak*-closed maximal left ideals, none of which is generated by a projection. We do not know if they can be finitely-generated.

\begin{example}		\label{eg4.13c}
	Let $G = \R^n \rtimes SO(n) \ (n \geq 2)$. Then $G$ is Hermitian by \cite[e.g. (a), pg. 191]{L}. It is known (see e.g. \cite[Example 4.42]{KT}) that the irreducible representations of $G$ fall into two kinds: those factoring through $SO(n)$, and those which are induced from irreducible representations of $\R^n \rtimes SO(n-1)$; the latter are the infinite-dimensional irreducible representations, and they vanish at infinity by \cite[Theorem 7.6]{KT}. However, $G$ has no integrable irreducible representations: this follows from  \cite[Corollary 2]{B80}, and the fact that $\widehat{G}$ with the Fell topology has no open points. It follows that $M(G)$ has no maximal left ideals generated by a projection by Proposition \ref{4.11}. As such, Proposition \ref{4.13} and Theorem \ref{4.4a} together imply that for each $\pi \in \widehat{G} \setminus \widehat{SO(n)}$ there exists $\xi \in \mathcal{H}_\pi$ such that $\mathcal{J}_{\pi, \xi}$ is a weak*-closed maximal left ideal, but is not generated by a projection. 
\end{example}

The next example contrasts with the previous one.

\begin{example}		\label{e.g.4.13e}
	Let $H$ denote the Heisenberg group
	$$H=\left\{ 
	\begin{pmatrix}
	1 & x & z \\
	0 & 1 & y \\
	0 & 0 & 1
	\end{pmatrix}
	: x,y,z \in \R
	\right\},$$
	 and let $N$ be the subgroup with $x = y =0$ and $z \in \Z$. Let $G$ be $H/N$, which is sometimes called the Weyl-Heisenberg group. This group is a connected nilpotent Lie group, and as such Theorem \ref{4.4a} applies. Identify $G$ with the set $\R \times \R \times \T$, with the multiplication given by
	 $$(x_1,y_1,a_1)(x_2,y_2,a_2) = (x_1+x_2,y_1+y_2, a_1a_2\e^{ix_1y_2}).$$
	 The infinite dimensional irreducible representations of $G$ act on $L^2(\R)$ and have the form 
	 $$[U_n(x,y,a)f](s) =a^n\e^{-insy}f(s-x) \qquad (f \in L^2(\R))$$
	 for $n \in \Z \setminus \{ 0 \}$. Each $U_n$ is integrable (for instance $\xi*_{U_n}\xi \in L^1(G)$, where $\xi(s) = \e^{-s^2}$), and as such vanishes at infinity, so that Theorem \ref{4.4a} tells us that the weak*-closed maximal left ideals of $M(G)$ have the form $\mathcal{J}_{U_n,\xi}$, where $n \in \Z\setminus \{ 0 \}$ and $\xi \in S_{\mathcal{H}_{U_n}}$. In fact we can say more: by \cite{V84}, for every $n \in \Z \setminus \{ 0 \}$ there exists a dense subspace $E \subset L^2(\R)$ such that, for every $\xi \in E$, the function $\xi*_{U_n}\xi$ can be scaled to give a minimal projection in $L^1(G)$.  Denoting this projection by $p$, it follows from the proof of Proposition \ref{4.11} that $\mathcal{J}_{{U_n},\xi} = M(G)*(\delta_e - p)$, and that it is a maximal left ideal. 
	 We do not know whether all of the weak*-closed maximal left ideals are generated by projections.
\end{example}

%\begin{example}		\label{e.g.4.13d}
%	Fix a prime number $p$ and let $\mathcal{O}_p$ denote the compact subgroup of the p-adic numbers $\Q_p$ given by $\{ x\in\Q_p : |x|_p = 1 \}$. Let $G$ be the Fell group $\Q_p \rtimes \mathcal{O}_p$. The irreducible representations of $G$ fall into two kinds: those factoring through $\mathcal{O}_p$ and those induced from $\Q_p$; the latter are integrable, whereas the former do not vanish at infinity \textbf{[cite]}. Hence by Theorem \ref{4.4a} the weak*-closed maximal left ideals of $M(G)$ have the form $\mathcal{J}_{\pi,\xi}$, where $\pi$ is a representation induced from an character on $\Q_p$, and $\xi \in S_{\mathcal{H}_\pi}$. In fact we can say more: by \cite{V84}, for every integrable $\pi \in \widehat{G}$ there exists a dense subspace $\widetilde{\mathcal{H}_\pi} \subset \mathcal{H}_\pi$ such that, for every $\xi \in \widetilde{\mathcal{H}_\pi}$, the function $\xi*_\pi\xi$ can be scaled to give a minimal projection in $L^1(G)$.  Denoting this projection by $p$, it follows from the proof of Proposition \ref{4.11} that $\mathcal{J}_{\pi,\xi} = M(G)*(\delta_e - p)$, and that it is a maximal left ideal. 
%	We do not know whether all of the weak*-closed maximal left ideals are generated by projections.
%\end{example}

We have seen in Proposition \ref{1.4} that a group with non-compact centre cannot have any weak*-closed maximal left ideals in its measure algebra, and Corollary \ref{4.9c}(ii) gives an even stronger statement for Hermitian groups. The next example is of a class of groups, not covered by these previous results, but which nonetheless has no weak*-closed maximal left ideals in its measure algebra. 

\begin{example} 		\label{eg4.13b}
	Consider $G = \R^2 \rtimes_A \R$, where $A \in M_2(\R),$ and $\R$ acts on $\R^2$ via $$t \colon x \rightarrow {\rm e}^{tA}x \qquad (t \in \R, \ x \in \R^2).$$ We claim that for these groups $M(G)$ has no weak*-closed maximal left ideals. 
	
	In \cite[Chapter 4]{KT} the Mackey Machine is employed to classify the irreducible representations of the groups $\R^2 \rtimes_A \R \ (A \in M_2(\R))$ according to the Jordan normal form of $A$. In the case that $A$ is similar to a matrix of the form $\big(\begin{smallmatrix} 0 & \lambda \\ -\lambda &0 \end{smallmatrix}\big),$ for some $\lambda \in \R \setminus \{ 0 \}$, the centre of $G$ contains a copy of $\Z$, so in this case the claim follows by Proposition \ref{1.4}. 
	
	We shall focus on the remaining cases, where $A$ is not similar to a matrix of the form $\big(\begin{smallmatrix} 0 & \lambda \\ -\lambda &0 \end{smallmatrix}\big)$. In these cases the elements of $\widehat{G}$ either factor through $\R$, or have the form  $U^\gamma \colon G \rightarrow \B(L^2(\R))$, where	$\gamma \in \R^2$, and $U^\gamma$ is given by
	\begin{equation}		\label{eq6.1}
		[U^{\gamma}(x, t)f](s) = \e^{i\gamma \cdot \e^{-sA}x}f(s-t)
		\qquad (s \in \R, \ f \in L^2(\R), \ (x, t) \in \R^2 \rtimes_A \R).
	\end{equation}

	Note that $G$ satisfies the hypothesis of Theorem \ref{4.4a}: the representations appearing \eqref{eq6.1} are induced from $\R^2$, and vanish at infinity modulo their kernels by \cite{BT82}. Of course, in the case of a representation that factors through the subgroup $\R$, the projective kernel is the entire group $G$, so that these representations vanish at infinity modulo their projective kernels for trivial reasons. Finally, $G$ is Hermitian by \cite{Lep77}. We shall show that none of the irreducible representations vanish at infinity, and hence, by Theorem \ref{4.4a}, $M(G)$ has no weak*-closed maximal left ideals.
	
	We only need to consider the infinite-dimensional irreducible representations. We shall divide into two cases.
	
	\underline{Case 1.} Suppose that $\{ \gamma \}^\perp$ is invariant for the action of 
	$\R$ on $\R^2$ via $A$. Then, taking $f = \chi_{[0,1]},$ we see that
	$$\langle U^\gamma (x, 0)f, f \rangle = \int_{-\infty}^{\infty} \chi_{[0,1]}(s) \dd s
	= 1,$$
	for all $x \in \{ \gamma \}^\perp$.  Hence $f*_{U^\gamma} f \notin C_0(\R^2 \rtimes_A \R)$.
	
	\underline{Case 2.} Suppose that $\{ \gamma \}^\perp$ is not invariant for the action of 
	$\R$ on $\R^2$ via $A$. Fix $x \in \{ \gamma \}^\perp \setminus \{ 0 \}$ and $a \in \R$ such that $\gamma \cdot \e^{-aA}x \neq 0$. By scaling $x$ we may assume that $\gamma \cdot \e^{-aA}x = 1$. Define $f = \chi_{[0, a]}$, and $g(s) = -\gamma \cdot (A\e^{-sA}x)\chi_{[0,a]}(s) \ (s \in \R)$, and note that 
	$$\frac{\rm d}{{\rm d}s} \gamma \cdot \e^{-sA}x = -\gamma \cdot (A\e^{-sA}x).$$
	Then for all $k \in \N$ we have
	\begin{align*}
	\langle U^\gamma ( (2k+1) \pi x, 0)f, g \rangle &= \int_0^{a} \e^{ (2k+1) \pi  i\gamma \cdot \e^{-sA}x}(- (2k+1) \pi \gamma \cdot (A\e^{-sA}x)) \dd s \\
	&= -i[\e^{ (2k+1) \pi  i\gamma \cdot \e^{-sA}x}]^a_0 = -i(\e^{(2k+1)\pi i} - e^0) =2i.
	\end{align*}
	Hence $f*_{U^\gamma} g \notin C_0(\R^2 \rtimes_A \R)$, and this concludes the argument.
	
	Finally, we note that many of these groups are not covered by Corollary \ref{4.9c}(ii). Indeed, it is routinely shown that these groups have no proper finite index subgroups, and have trivial centre whenever the action determined by $A$ is faithful and has no non-zero fixed points. This happens, for instance, for $A = \big(\begin{smallmatrix} 1 & 0\\ 0 & -1 \end{smallmatrix}\big)$ or 
	$A = \big(\begin{smallmatrix} \lambda & 1\\ 0 & \lambda \end{smallmatrix}\big)$, 
	for $\lambda \in \R \setminus \{ 0 \}$.
\end{example}

\section{Finitely-Generated Closed Left Ideals of Measure Algebras}
\noindent

In this section we shall show how the results that we have proved so far imply Theorem \ref{0.3}. We shall then take a more detailed look at the weak*-closed left ideals of $M(\R^2 \rtimes SO(2))$ (see Example \ref{eg4.13c} above), and in particular prove Theorem \ref{0.4}.

\begin{proof}[Proof of Theorem \ref{0.3}]
	(i) Let $G$ be a Moore group and suppose that $M(G)$ has a maximal left ideal $J$ that is weak*-closed. Moore groups are Hermitian by \cite[diagram, pg. 1486]{P01}.
	By Corollary \ref{4.3c} there exists $\pi \in \widehat{G}$ and $\xi \in S_{\mathcal{H}_\pi}$ such that $J = \mathcal{J}_{\pi, \xi}$.
	Since $\pi$ is finite-dimensional, $\mathcal{J}_{\pi, \xi}$ has finite codimension and hence $G$ must be compact by Corollary \ref{1.3}.
	
	(ii) This is immediate from Proposition \ref{1.4}.
	
	(iii) $G$ has a finite-index normal nilpotent subgroup $N$, which is also finitely-generated. By \cite[5.2.22(ii)]{Rob} $Z(N)$ is infinite. The result now follows from Corollary \ref{4.9c}(ii).
	
	(iv) This is Example \ref{eg4.13b}.
\end{proof}

We now turn to the proof of Theorem \ref{0.4}. In very broad terms the idea is to show that if $\mathcal{J}_{U,1}$ is generated by finitely many compactly-supported measures, then measures belonging to the ideal decay unexpectedly rapidly in a certain sense, and this will give a contradiction.

For the remainder of this section we shall set $G = \R^2 \rtimes SO(2)$. For notational convenience we shall sometimes formally identify $SO(2)$ with $[0, 2\pi)$.  This allows us to represent an element of $G$ as $(x, \theta)$, where $x \in \R^2$, and $\theta \in [0,2\pi)$. 

We shall write $B_0(\rho)$ for the ball of radius $\rho>0$ about the origin in $\R^2$.
 Let $B_1 \subset G$ be a set of the form $B_0(\rho) \times SO(2)$, and let
$$B_n = \{ t_1\cdots t_n : t_1, \ldots, t_n \in B_1 \}.$$
Then $B_n = B_0(n\rho) \times SO(2)$.
We shall also write $S_n = B_n \setminus B_{n-1}$.
Let $\chi$ be the character of $\R^2$ given by $\chi(x) = \e^{ix_1} \ (x \in \R^2)$, and let $U \in \widehat{G}$ be the representation on $L^2(SO(2))$ given by
$$[U(x, \theta) \xi](\varphi) = \exp(i(x_1 \cos\varphi+ x_2 \sin\varphi))\xi(\varphi-\theta),$$ 
for $x \in \R^2, \ \theta,\varphi \in [0, 2\pi),$ and $\xi \in L^2(SO(2)).$ In fact, $U$ is the representation of $G$ induced by $\chi$ and vanishes at infinity. We shall study $\mathcal{J}_{U,1}$.

\begin{lemma}		\label{5.3}
	Fix $\rho > 0$, and let $B_1$ and $B_n \ (n \in \N)$ be defined in terms of $\rho$ as above. Let 
	$$\sigma_n (\mu) = \int_{B_n} \langle U(t) 1, 1 \rangle \dd \mu (t) \qquad (n \in \N, \ \mu \in \mathcal{J}_{U, 1}).$$
	 Let $\nu \in \mathcal{J}_{U, 1}$ satisfy $\supp \nu \subset B_1$. Then for every $\mu \in M(G)$ we have 
	$$\sum_{n=1}^\infty \vert \sigma_n(\mu*\nu) \vert < \infty.$$
\end{lemma}

\begin{proof}
	We calculate 
	\begin{align*}
	\sigma_n(\mu*\nu) &= \int_{B_n} \langle U(t)1, 1 \rangle \dd(\mu*\nu)(t) 
	= \int_G \int_G \1_{B_n}(st) \langle U(st)1, 1 \rangle \dd \mu(s) \dd \nu (t) \\
	&= \int_{B_1} \int_{B_n t^{-1}} \langle U(st)1, 1 \rangle \dd \mu(s) \dd \nu(t) = I_1 + I_2,
	\end{align*}
	where
	$$I_1 = \int_{B_1} \int_{B_{n-1}} \langle U(st) 1, 1 \rangle \dd \mu(s) \dd \nu(t),$$
	and
	$$I_2 = \int_{B_1} \int_{B_nt^{-1} \setminus B_{n-1}} \langle U(st) 1, 1 \rangle \dd \mu(s) \dd \nu(t).$$
	Now,
	$$I_1 = \int_{B_{n-1}} \int_{B_1} \langle U(st) 1, 1 \rangle \dd \nu(t) \dd \mu(s) = \int_{B_{n-1}} \langle U(s) U(\nu)1, 1 \rangle \dd \mu(s) = 0,$$
	because $\supp \nu \subset B_1$ and $U(\nu) 1 = 0$.
	%Here, the change in the order of integration is justified by Fubini's Theorem, since
	%$$\int_{B_{n-1}} \int_{B_1} \vert \langle U(st)1, 1 \rangle \vert \dd \vert \nu \vert(t) \dd \vert \mu \vert(s) \leq \Vert \nu \Vert \Vert \mu \Vert < \infty.$$
	Hence 
	\begin{align*}
	\vert \sigma_n(\mu*\nu) \vert &\leq \vert I_2 \vert 
	\leq \int_{B_1} \int_{B_nt^{-1} \setminus B_{n-1}} \vert \langle U(st)1, 1 \rangle \vert \dd \vert \mu \vert(s) \dd \vert \nu \vert(t) \\
	&\leq \Vert \nu \Vert \vert \mu \vert (B_{n+1} \setminus B_{n-1}) = \Vert \nu \Vert \vert \mu \vert (S_{n+1} \cup S_n),
	\end{align*}
	since $\vert \langle U(st) 1,1 \rangle \vert \leq 1 \ (s,t \in G)$ and $B_nt^{-1} \setminus B_{n-1} \subset B_{n+1} \setminus B_{n-1}$.
	Hence
	$$\sum_{n=1}^\infty \vert \sigma_n(\mu*\nu) \vert \leq \Vert \nu \Vert \sum_{n=1}^\infty \vert \mu \vert(S_{n+1}) + \Vert \nu \Vert \sum_{n=1}^\infty \vert \mu \vert(S_n) \leq 2 \Vert \nu \Vert \Vert \mu \Vert < \infty,$$
	as required.
\end{proof}

  From now on we shall write $\chi^\varphi(x) = \e^{i(x_1\cos\varphi - x_2\sin\varphi)} \ (x \in \R^2)$, that is $\chi$ applied to $x$ rotated anticlockwise through and angle of $\varphi$.

\begin{lemma}		\label{5.4a}
	For every $x \in \R^2$ we have 
	$$[U(x,0)1](\varphi) = \chi^{-\varphi}(x) \qquad \varphi \in [0, 2\pi).$$
\end{lemma}

\begin{proof}
	Direct computation based on the formula for an induced representation.
\end{proof}

\begin{lemma}		\label{5.4}
	Let $k \in \N$, and set $\rho_n = 2\pi k n \ (n \in \N)$. Let $B_1 = B_0(2 \pi k)$, so that $B_n = B_0(\rho_n) \times SO(2) \ ( n \in \N)$, and let $\sigma_n$ be defined as in Lemma \ref{5.3}.
	Then there exists $\nu \in \mathcal{J}_{U,1}$ such that $\sum_{n=1}^\infty \vert \sigma_n(\nu) \vert$ diverges. 
\end{lemma}

\begin{proof}
	Write $S_0(\rho_n) = \{ x \in \R^2 : \Vert x \Vert = \rho_n \}$. Note that $\Vert \1_{S_0(\rho_n)} \Vert = 2\pi \rho_n = 4\pi^2kn.$
	We claim that 
	\begin{equation}		\label{eq5.4.1}
	\chi(\1_{S_0(\rho_n)}) = \Vert \1_{S_0(\rho_n)} \Vert J_0(\rho_n) \qquad (n \in \N).
	\end{equation}
	Indeed,
	\begin{align*}		
	\chi(\1_{S_0(\rho_n)}) &= \int_{S_0(\rho_n)} \e^{ix_1} \dd x 
	= \int_{S_0(\rho_n)} \cos(x_1) \dd x
	= 4 \int_0^{\frac{\pi}{2}} \rho_n \cos(\rho_n\cos(\theta)) \dd \theta \\
	&= 4 \times \frac{\pi \rho_n J_0(\rho_n)}{2} 
	= 2\pi \rho_n J_0(\rho_n),
	\end{align*}
	where $J_0$ denotes the zeroth order Bessel function of the first kind, and we have used the identity
	$$J_0(z) = \frac{1}{\pi} \int_0^\pi \cos(z\sin \tau) \dd \tau \qquad (z \in \C)$$
	to compute the integral. Using the following formula
	$$J_0(z) = \sqrt{\frac{2}{\pi z}}\left[ \cos \left(z - \pi/4 \right) + \e^{\vert \im z \vert} O(\vert z \vert^{-1}) \right] \qquad (\vert \arg  z \vert < \pi)$$
	we see that 
	\begin{equation}		\label{eq5.4.2}
	J_0(\rho_n) = 
	%\sqrt{\frac{1}{\pi^2kn}}\left[ \cos(2\pi kn - \pi/4) + O(\rho_n^{-1}) \right] =
	 \frac{1}{\pi \sqrt{2kn}} + O(1/n^{3/2}).
	\end{equation}
	
	Define measures $\nu_n \in M(\R^2)$ for $n \in \N$ by 
	$$\nu_n = \frac{1}{n^{3/2}\Vert \1_{S_0(\rho_n)} \Vert} \1_{S_0(\rho_n)} \qquad (n \in \N)$$
	and define 
	$$\nu_0 = \zeta \delta_e - \sum_{n=1}^\infty \nu_n,$$
	where $\zeta := \chi \left(\sum_{n=1}^\infty \nu_n \right)$. Note that $\nu_0$ is well-defined since the $\nu_n$'s are disjointly supported, and $\Vert \nu_n \Vert = n^{-3/2}$. Since $\nu_0$ is rotationally symmetric, $\chi^\varphi(\nu_0) = \chi(\nu_0)$ for all $\varphi \in [0, 2 \pi)$. Finally, we define $\nu = \nu_0 \times \delta_0 \in M(\R^2 \rtimes SO(2))$. First we show that $\nu \in \mathcal{J}_{U,1}$. Given $\eta \in L^2(SO(2))$ we have
	\begin{align*}
	\langle U(\nu)1,\eta \rangle &= \int_{SO(2)} \int_{SO(2)} \int_{\R^2} [U(x, \theta)1](\varphi)\overline{\eta(\varphi)} \dd \nu_0(x) \dd \delta_0(\theta) \dd \varphi \\
	&= \int_{SO(2)} \int_{\R^2} [U(x,0)1](\varphi)\overline{\eta(\varphi)} \dd \nu_0(x) \dd \varphi \\
	&= \int_{SO(2)} \int_{\R^2} \chi^{-\varphi}(x) \overline{\eta(\varphi)} \dd \nu_0(x) \dd \varphi \qquad (\text{by Lemma \ref{5.4a}}) \\
	&= \left(\int_{SO(2)} \overline{\eta(\varphi)} \dd \varphi \right) \left(\int_{\R^2} \chi(x) \dd \nu_0(x) \right) \qquad (\text{because } \nu_0 \text{ is rotationally symmetric}) \\
	&=\chi(\nu_0)\int_{SO(2)} \overline{\eta(\varphi)} \dd \varphi =0,
	\end{align*}
	because $\chi(\nu_0) = 0$. Since $\eta$ was arbitrary $U(\nu)1 = 0$, i.e. $\nu \in \mathcal{J}_{U,1}$.
	
	Next we show that 
	\begin{equation}	\label{eq5.4.3}
	\sigma_n(\nu) = \sum_{m = n+1}^\infty \chi(\nu_m).
	\end{equation}
	Indeed,
	\begingroup
	\allowdisplaybreaks
	\begin{align*}
	\langle U(\nu_0 \times \delta_0 \vert_{B_n}) 1, 1 \rangle 
	&= \int_{SO(2)} \int_{\R^2} \langle U(x, \theta)1, 1 \rangle \dd (\nu_0 \times \delta_0 \vert_{B_n})(x) \\
	&= \int_{B_0(\rho_n)} \langle U(x,0)1,1 \rangle \dd \nu_0(x) 
	= \zeta \langle 1,1 \rangle - \sum_{m=1}^\infty \int_{B_0(\rho_n)} \langle U(x, 0)1, 1 \rangle \dd \nu_m(x).	
	%\qquad (\text{since } B_0(\rho_n) \cap \supp \nu_m = \emptyset \text{ for } m>n.) 
	\end{align*}
	Since $B_0(\rho_n) \cap \supp \nu_m = \emptyset$ for $ m>n$, this is equal to 
	\begin{align*}
	\zeta - \sum_{m=1}^n \int_{B_0(\rho_n)} &\langle U(x, 0)1, 1 \rangle \dd \nu_m(x)
	=\zeta - \sum_{m=1}^n \int_{\R^2} \int_{SO(2)} \chi^{-\varphi}(x) \dd \varphi \dd \nu_m(x) \\
	&= \zeta - \sum_{m=1}^n \int_{SO(2)} \int_{\R^2} \chi^{-\varphi}(x) \dd \nu_m(x) \dd \varphi \\
	&= \zeta - \sum_{m=1}^n \int_{SO(2)} \chi^{-\varphi}(\nu_m) \dd \varphi 
	= \zeta - \sum_{m=1}^n \int_{SO(2)} \chi(\nu_m) \dd \varphi \\
	&= \zeta - \sum_{m=1}^n \chi(\nu_m)
	= \sum_{m = n+1}^\infty \chi(\nu_m),
	\end{align*}
	\noindent
	where we have used Lemma \ref{5.4a} in the first line, and the fact that each $\nu_m$ is rotationally symmetric in the penultimate line.
	This establishes \eqref{eq5.4.3}.  By \eqref{eq5.4.1} and \eqref{eq5.4.2} we have 
	\begin{align*}
	\chi(\nu_n) =
	\frac{1}{n^{3/2}}J_0(\rho_n) = \frac{1}{\pi \sqrt{2k} n^2} +O(n^{-3}),
	\end{align*}
	for all $n \in \N$, so there exists a constant $C>0$ such that 
	$$\left\vert \chi(\nu_n) - \frac{1}{\pi \sqrt{2k} n^2} \right\vert  \leq \frac{C}{n^3} \qquad (n \in \N).$$
	It follows that for all $n \in \N$ we have 
	$$\sum_{m=n+1}^\infty \chi(\nu_m) \geq \sum_{m=n+1}^\infty 
	\frac{1}{\pi \sqrt{2k} m^2} - C \sum_{m=n+1}^\infty \frac{1}{m^3}.$$
	By an integral estimate 
	$$\sum_{m=n+1}^\infty \frac{1}{m^3} \leq \frac{1}{(n+1)^{3}} + \frac{1}{2(n+1)^2},$$
	whereas
	$$\sum_{m=n+1}^\infty \frac{1}{m^2} \geq \frac{1}{n},$$
	for $n \in \N$.
	Combining these estimates with \eqref{eq5.4.3} gives
	$$\sum_{n=1}^N \vert \sigma_n(\nu) \vert \geq \frac{1}{\pi \sqrt{2k}} \sum_{n=1}^N \frac{1}{n} - C \sum_{n=1}^N \frac{1}{(n+1)^3} + \frac{1}{2(n+1)^2},$$
	which tends to infinity as $N \rightarrow \infty$.
	\endgroup
\end{proof}

We can now prove Theorem \ref{0.4}.

\begin{proof}[Proof of Theorem \ref{0.4}]
	Part (i)  follows from Example \ref{eg4.13c}. We shall prove part (ii). As $U$ vanishes at infinity, $\mathcal{J}_{U,1}$ is weak*-closed by Corollary \ref{4.4}(i). We shall prove that $\mathcal{J}_{U,1}$ cannot be generated by finitely many compactly supported elements. Let $G = \R^2 \rtimes SO(2)$, and suppose that for some $m \in \N$ there are compactly-supported measures $\mu_1, \ldots, \mu_m \in \mathcal{J}_{U,1}$ such that 
	$$\mathcal{J}_{U,1} = M(G)*\mu_1 + \cdots + M(G)*\mu_m.$$
	There exists $k \in \N$ such that $\supp \mu_i \subset B_0(2\pi k) \times SO(2) \ (i = 1, \ldots, m)$. 
	It follows from Lemma \ref{5.3} with $\rho = 2\pi k$ that $\sum_{n=1}^\infty \vert \sigma_n(\mu) \vert < \infty$ for every $\mu \in \mathcal{J}_{U,1}$. This contradicts Lemma \ref{5.4}.
\end{proof}

\textit{Remark.} Continue to write $G = \R^2 \rtimes SO(2)$, and let $M_c(G)$ denote the compactly supported measures in $M(G)$. If we could prove that $M_c(G) \cap \mathcal{J}_{U,1}$, or even $M(G)*(M_c(G) \cap \mathcal{J}_{U,1})$, was norm-dense in $\mathcal{J}_{U,1}$, then \cite[Lemma 2.1]{W1}, together with Theorem \ref{0.4}, would imply that $\mathcal{J}_{U,1}$ is not finitely-generated, and hence that finite generation and weak*-closedness are distinct notions for closed left ideals of measure algebras. However, we have not been able to prove this.

%\section{Finitely-Generated Maximal Left Ideals of Beurling Algebras}
%\noindent
%\textbf{[I could decide not to include this for now.]}

\section{Open Questions}
\noindent
We conclude this article with some problems for future study. The main question we would like to resolve is whether or not Conjecture \ref{0.1}  and Conjecture \ref{0.2} hold for all Hermitian locally compact groups. Another of the central topics of this article is the connection between finite generation of a closed left ideal of a measure algebra, and weak*-closedness. This suggests to us the following natural questions that we have been unable to answer.

\begin{question}	\label{Q1}
	\begin{enumerate}
		\item[\rm (i)] Does there exist a locally compact group $G$, and a closed left ideal $J$ in $M(G)$, such that $J$ is weak*-closed, but not finitely-generated?
		
		\item[\rm (ii)] Does there exist a locally compact group $G$, and a closed left ideal $J$ in $M(G)$, such that $J$ is finitely-generated but not generated by a projection?
	\end{enumerate}
\end{question} 

It would be particularly interesting to be able to answer these questions for maximal left ideals. 

In \cite{W1} it was shown that, if $G$ is a non-discrete locally compact group, then $L^1(G)$ has no finitely-generated maximal modular left ideals. It would be interesting to know whether the same holds for an infinite discrete group $G$. We saw in Theorem \ref{0.3} that this does hold when $G$ is finitely-generated and virtually nilpotent.

\begin{question}		\label{Q2}
	\begin{enumerate}
		\item[\rm (i)] Does $\ell^1(G)$ ever have a finitely-generated maximal left ideal, for an infinite group $G$?
		\item[\rm (ii)] Can $\ell^1(G)$ ever have a weak*-closed maximal left ideal, for an infinite group $G$?
    \end{enumerate}
\end{question}

Another class of dual Banach algebras associated with locally compact groups are the Fourier-Stieltjes algebras $B(G)$. These are commutative algebras, so the maximal ideals automatically have codimensions 1, and finding the  weak*-closed maximal ideals corresponds to finding the weak*-continuous characters on $B(G)$. As such, in the case that $G$ is amenable an answer to this question is given as a special case of a theorem of Ilie and Stokke \cite[Theorem 5.11]{IS}. However, without the assumption of amenability the question remains open.
%Since the predual of $B(G)$ is given by $C^*(G)$ it is clear that, for a discrete group $G$, and $t_0 \in G$, the set $\{u \in B(G) : u(t_0) = 0 \}$ is always a weak*-closed maximal ideal of $B(G)$.

\begin{question} 		\label{Q3}
Let $G$ be a locally compact group. What are the weak*-closed maximal ideals of $B(G)$?
\end{question}

%\begin{question}		\label{Q3}
%	\begin{enumerate}
%		\item[\rm (i)] Can the Fourier-Stieltjes algebra $B(G)$ ever have a weak*-closed maximal ideal, for a non-discrete locally compact group $G$?
		
%		\item[\rm (ii)] Let $G$ be a discrete group. Is it true that every weak*-closed maximal ideal of $B(G)$ is of the form $\{u \in B(G) : u(t_0) = 0 \}$ for some $t_0 \in G$?
%	\end{enumerate}
%\end{question}

It would also be interesting to study the weak*-closed ideals of the reduced Fourier-Stieltjes algebra $B_r(G)$ and the algebra of completely bounded multipliers of the Fourier algebra $M_{cb}A(G)$.  In all of these examples the absence of a bounded approximate identity in the Fourier algebra $A(G)$ for non-amenable groups introduces new difficulties not seen in the present work.

\subsection*{Acknowledgements}
Part of this work was undertaken during my postdoctoral position at the Laboratoire des Math{\' e}matiques de Besan{\c c}on, and I would like to thank Uwe Franz and Yulia Kuznetsova for their kind hospitality during that period. I would like to thank Simeng Wang, Keith Taylor, and Michael Leinert for useful email exchanges. I am indebted to Bence Horv{\'a}th for his careful reading of the manuscript. Finally, I want to thank the anonymous referees for their helpful and insightful comments, and in particular for pointing out how to simplify the proof of Lemma \ref{1.1}.

\end{document}